\documentclass[final,1p,times,margin=1in]{elsarticle}  

\usepackage{amssymb}
 \usepackage{amsthm}
\usepackage{amscd}
\usepackage{amsmath}
\usepackage{amsfonts}
\usepackage{amssymb}
\usepackage{color}
\usepackage{graphicx}
\usepackage{url}
\newtheorem{theorem}{Theorem}[section]

\newtheorem{lemma}[theorem]{Lemma}

\newtheorem{remark}[theorem]{Remark}
\usepackage[ruled]{algorithm2e}
\usepackage{mathrsfs}
\usepackage{titletoc}
\usepackage{float}
\usepackage{float}
\usepackage{amsmath}

\usepackage{tikz}
\usepackage{caption}
\usepackage{subcaption}
\usepackage{booktabs}
\usepackage{geometry}

\journal{******}

\begin{document}

\begin{frontmatter}

\title{Equivalence of Lin--Lu--Yau curvature and 1/2-Ollivier curvature on weighted graphs}

\author[ustc]{Shiping Liu}
\ead{spliu@ustc.edu.cn}

\author[ruc]{Yunyan Yang}
\ead{yunyanyang@ruc.edu.cn}

\address[ustc]{School of Mathematical Sciences, University of Science and Technology of China,
Hefei, 230026, China}
\address[ruc]{School of Mathematics, Renmin University of China, Beijing, 100872, China}

\begin{abstract}
In this note, we prove that, on weighted graphs, the Lin--Lu--Yau curvature coincides with the $p$-Ollivier curvature up to scaling whenever the idleness parameter $p\geq 1/2$. Moreover, the threshold $1/2$ is sharp.
This extends an earlier result of Bourne et al. (Ollivier--Ricci idleness functions
of graphs, SIAM J. Discrete Math., 32 (2018), no. 2, 1408-1424), where combinatorial graphs were considered.
This observation yields a simple proof for the global existence and uniqueness of solutions of the Lin--Lu--Yau curvature flow in Bai et al. (Ollivier Ricci-flow on weighted graphs, Amer. J. Math. 146 (2024), 1723-1747).
\end{abstract}

\begin{keyword}
 Ollivier curvature  \sep Lin--Lu--Yau curvature\sep weighted graphs \sep curvature flow
\MSC[2020] 05C99\sep 05C22\sep 47A10 \sep 49Q20
\end{keyword}

\end{frontmatter}

\titlecontents{section}[0mm]
                       {\vspace{.2\baselineskip}}
                       {\thecontentslabel~\hspace{.5em}}
                        {}
                        {\dotfill\contentspage[{\makebox[0pt][r]{\thecontentspage}}]}
\titlecontents{subsection}[3mm]
                       {\vspace{.2\baselineskip}}
                       {\thecontentslabel~\hspace{.5em}}
                        {}
                       {\dotfill\contentspage[{\makebox[0pt][r]{\thecontentspage}}]}

\setcounter{tocdepth}{2}



\numberwithin{equation}{section}
\section{Introduction}

Let $G=(V,E)$ be a locally finite graph, where $V$ denotes the vertex set and $E$ denotes the edge set. Consider the probability measures $\tilde{\mu}_x^p$ for any $x\in V$, $p\in [0,1]$:
$$\tilde{\mu}_x^p(z)=\left\{
\begin{array}{lll}
p&{\rm if}& z=x\\[1.2ex]
\frac{1-p}{d_x}&{\rm if}& z\sim x\\[1.2ex]
0&{\rm if}&{\rm otherwise},
\end{array}\right.
$$
where $z\sim x$ means $z$ is a neighbor of $x$ and the degree $d_x$ stands for the number of all neighbors of $x$.
Let $W_1$ be the Wasserstein distance between two probability measures on $V$. Then the
$p$-Ollivier curvature of an edge $xy$ reads as
$$\tilde{\kappa}_p(x,y)=1-W_1(\tilde{\mu}_x^p,\tilde{\mu}_y^p).$$
Here the distance between two vertices of any edge is assumed to be $1$. And the Lin--Lu--Yau curvature \cite{Lin--Lu--Yau} is defined by
$$
\tilde{\kappa}_{\rm LLY}(x,y)=\lim_{p\rightarrow 1-0}\frac{\tilde{\kappa}_p(x,y)}{1-p}.
$$
It was proved by Bourne, Cushing, Liu, M\"unch and Peyerimhoff \cite{BCLMP} that if
\begin{equation}\label{p-0}\frac{1}{1+\max\{d_x,d_y\}}\leq p\leq 1,\end{equation} 
then the identity
\begin{equation}\label{relation-0}\tilde{\kappa}_p(x,y)=(1-p)\tilde{\kappa}_{\rm LLY}(x,y)\end{equation}
holds on each edge $xy\in E$.
In the special case that $G$ is regular, i.e. $d_x=d$ for all $x\in V$ and some fixed positive integer $d$, there is an equality
\begin{equation}\label{kappa-0}
\tilde{\kappa}_{\rm LLY}(x,y)=2\tilde{\kappa}_{\frac{1}{2}}(x,y),
\end{equation}
for each edge $xy\in E$. The identity \eqref{kappa-0} has been extended to long-scale curvature, that is, the Ollivier and Lin--Lu--Yau curvature of a pair of vertices at arbitrary combinatorial distance \cite{Cushing-Kamtue}. 

The aim of this note is to extend (\ref{relation-0}) and (\ref{kappa-0}) to arbitrary locally finite weighted graphs.
For this purpose, we assume that $G=(V,E,w)$ is a connected locally finite weighted graph, where $w_{xy}>0$ for all edges $xy\in E$.
If $w_{xy}=1$ for all edges $xy\in E$, then $G$ is reduced to a combinatorial graph. In other words, combinatorial graph is
a special case of weighted graph. Let $\rho: V\times V\rightarrow [0,+\infty)$
be an arbitrarily chosen distance function. Consider a probability measure $\mu_x^p$ with idleness $p\in [0,1]$, such that
\begin{equation}\label{mpx}
\mu_x^p(z)=\left\{
\begin{array}{lll}
p & \text{if }& z=x\\[1.2ex]
(1-p)\mu_{{x}}^\ast(z) &{\rm if}& z\in V\setminus\{x\},
\end{array}\right.
\end{equation}
where $\mu_{{x}}^\ast$ is a probability measure on $V\setminus\{x\}$ with finite supports. Let
$x,y$ be any two vertices in $V$ ($y$ is not necessarily a neighbor of $x$).
Then we have
\begin{equation}\label{eq:W1}
W_1(\mu_x^p,\mu_y^p):=\inf_{\pi\in \Pi(\mu_x^p,\mu_y^p)}\sum_{x,y\in V}\rho(x,y)\pi(x,y)=\sup_{f\in {\rm Lip}\,1}\sum_{u\in V}f(u)(\mu_x^p(u)-\mu_y^p(u)),
\end{equation}
where $\Pi(\mu_x^p,\mu_y^p)$ denotes the set of all couplings of $\mu_x^p$ and $\mu_y^p$,
and ${\rm Lip}\,1$ stands for the set of all functions $f$ such that
$$|f(x)-f(y)|\leq \rho(x,y),\quad\text{for all}\ x,y\in V.$$
We call $\pi\in \Pi(\mu_x^p,\mu_y^p)$ and $f\in {\rm Lip}\,1$ the optimal transport plan and optimal Kantorovich potential transporting $\mu_x^p$ to $\mu_y^p$, respectively, if they achieve the $\inf$ and $\sup$ in \eqref{eq:W1}.
Since \(W_1(\mu_x^p,\mu_y^p)\) is the supremum of affine functions of $p$, the function \(p\mapsto W_1(\mu_x^p,\mu_y^p)\) is convex.
Define the \(p\)-Ollivier curvature \cite{Ollivier} as
\begin{equation}\label{Ollivier}\kappa_p(x,y):=1-\frac{W_1(\mu_x^p,\mu_y^p)}{\rho(x,y)},\end{equation}
and the Lin--Lu--Yau curvature \cite{Lin--Lu--Yau} as
\begin{equation}\label{LLY}\kappa_{\rm LLY}(x,y):=\lim_{p\to 1-0}\frac{\kappa_p(x,y)}{1-p}.\end{equation}
\begin{theorem}\label{thm:1-p}
    Let $G=(V,E,w)$ be a connected locally finite weighted graph associated with a distance function $\rho:V\times V\rightarrow[0,+\infty)$. Then, for any two vertices $x,y\in V$, \[\kappa_p(x,y)=(1-p)\kappa_{\rm LLY}(x,y)\]
    holds for all $p\in [1/2,1]$.
\end{theorem}
Lin, Lu, and Yau \cite{Lin--Lu--Yau} recognized and investigated the important role played by the idleness parameter \(p\) in the definition of Ollivier curvature.
The above theorem shows that, for \(p\geq 1/2\), Ollivier curvature and Lin--Lu--Yau curvature carry the same information. In particular,
$$\kappa_{\rm LLY}(x,y)=2\kappa_{1/2}(x,y),\quad\text{for all}\ x,y\in V.$$
In this way, we obtained analogs of (\ref{p-0})-(\ref{kappa-0}) in our setting. Note that the distance function $\rho$ may be either determined by edge weights, or independent of edge weights. Moreover, the threshold $p=1/2$ is sharp in the general setting of weighted graphs, see Remark \ref{rmk:sharp} below.\\

In the remaining part of this note, we prove Theorem \ref{thm:1-p} in Section 2, and give two general $p$-Ollivier curvature flows in Section 3. Our discussion about curvature flow provides a simple proof for the global existence and uniqueness of solutions of the Lin--Lu--Yau curvature flow in \cite{Bai-Lin-Lu-Wang-Yau}.

\section{The proof of Theorem \ref{thm:1-p}}
To prove the main theorem, we need the following two lemmas. The first one follows from the Complementary Slackness Theorem, see, for example, \cite[page 49]{Minoux} or \cite[page 88]{Villani}.
\begin{lemma} \label{lem:slackness} Let $G=(V,E,w)$ be a locally finite weighted graph, $x,y\in V$, $p\in [0,1]$,
$\mu_x^p$, $\mu_y^p$ be defined as in (\ref{mpx}), and $\rho:V\times V\rightarrow [0,+\infty)$ be a distance function. Suppose that $\pi$ and $\phi$
are optimal transport plan and optimal Kantorovich potential transporting $\mu_x^p$ to $\mu_y^p$, respectively. If two vertices $u,v$ satisfy  $\pi(u,v)\neq 0$, then we have
$\phi(u)-\phi(v)=\rho(u,v)$.
\end{lemma}
The lemma is an analog of \cite[Lemma 3.1]{BCLMP}. The only difference is that in our setting, the distance function $\rho$ and the random walk $\mu_x^p$
are more general. We will forward the proof adapted from there for completeness. \\

{\it Proof of Lemma \ref{lem:slackness}.}
The optimality of $\phi$ and $\pi$ implies
\begin{eqnarray}\label{Lip}&&|\phi(r)-\phi(s)|\leq \rho(r,s),\quad\forall r,s\in V,\\[1.2ex]
\label{transp}&&\sum_{r\in V}\pi(r,s)=\mu_y^p(s),\,\,\sum_{s\in V}\pi(r,s)=\mu_x^p(r),\quad \forall r,s\in V,\\
\label{W1}&& W_1(\mu_x^p,\mu_y^p)=\sum_{r\in V}\phi(r)(\mu_x^p(r)-\mu_y^p(r))=\sum_{r,s\in V}\pi(r,s)\rho(r,s).\end{eqnarray}
It follows from (\ref{Lip}), (\ref{transp}) and (\ref{W1}) that
\begin{eqnarray*}
W_1(\mu_x^p,\mu_y^p)&=&\sum_{r\in V}\phi(r)\mu_x^p(r)-\sum_{s\in V}\phi(s)\mu_y^p(s)\\
&=&\sum_{r\in V}\phi(r)\sum_{s\in V}\pi(r,s)-\sum_{s\in V}\phi(s)\sum_{r\in V}\pi(r,s)\\
&=&\sum_{r,s\in V}(\phi(r)-\phi(s))\pi(r,s)\\
&\leq&\sum_{r,s\in V}\rho(r,s)\pi(r,s)\\
&=&W_1(\mu_x^p,\mu_y^p).
\end{eqnarray*}
Hence
$$\sum_{r,s\in V}(\phi(r)-\phi(s))\pi(r,s)=\sum_{r,s\in V}\rho(r,s)\pi(r,s).$$
This together with (\ref{Lip}) and the fact $\pi(r,s)\in [0,1]$ yields the desired result. $\hfill\Box$

\begin{lemma}\label{lem:linear}
    Let $G=(V,E)$ be a connected locally finite weighted graph, $x,y\in V$, $0\leq p_1\leq p_2\leq 1$,
$\mu_x^p$, $\mu_y^p$ be defined as in (\ref{mpx}), and $\rho:V\times V\rightarrow [0,+\infty)$ be a distance function.
 If $\phi\in {\rm Lip}\,1$ is an optimal Kantorovich potential transporting $\mu_x^{p_1}$ to $\mu_y^{p_1}$ and transporting $\mu_x^{p_2}$ to $\mu_y^{p_2}$, simultaneously, then the function $p\mapsto W_1(\mu_x^p,\mu_y^p)$ is linear on $[p_1,p_2]$.
\end{lemma}
This is an analog of \cite[Lemma 4.2]{BCLMP}. We also adapt its proof to our setting for completeness.\\

{\it Proof of Lemma \ref{lem:linear}.} Let $0\leq p_1\leq p_2\leq 1$ and $p=tp_1+(1-t)p_2$ for any $t\in[0,1]$. Then $p\in[p_1,p_2]$.
In view of (\ref{mpx}), we have
$$\mu_x^p=t\mu_x^{p_1}+(1-t)\mu_x^{p_2},\quad \mu_y^{p}=t\mu_y^{p_1}+(1-t)\mu_y^{p_2}.$$
On one hand, the Kantorovich duality implies that for any $\psi\in {\rm Lip}\,1$,
\begin{eqnarray*}
\sum_{u\in V}\psi(u)(\mu_x^p(u)-\mu_y^p(u))&=&t\sum_{u\in V}\psi(u)(\mu_x^{p_1}(u)-\mu_y^{p_1}(u))+(1-t)\sum_{u\in V}\psi(u)(\mu_x^{p_2}(u)-\mu_y^{p_2}(u))\\
&\leq&tW_1(\mu_x^{p_1},\mu_x^{p_2})+(1-t)W_1(\mu_x^{p_2},\mu_y^{p_2}).
\end{eqnarray*}
Thus
\begin{equation}\label{onehand}W_1(\mu_x^{p},\mu_x^{p})\leq tW_1(\mu_x^{p_1},\mu_x^{p_2})+(1-t)W_1(\mu_x^{p_2},\mu_y^{p_2}).\end{equation}
On the other hand, since $\phi\in {\rm Lip}\,1$ is optimal simultaneously at $p_1$ and $p_2$, we have
\begin{eqnarray}\nonumber
W_1(\mu_x^p,\mu_y^p)&\geq&\sum_{u\in V}\phi(u)(\mu_x^p(u)-\mu_y^p(u))\\\nonumber
&=&t\sum_{u\in V}\phi(u)(\mu_x^{p_1}(u)-\mu_y^{p_1}(u))+(1-t)\sum_{u\in V}\phi(u)(\mu_x^{p_2}(u)-\mu_y^{p_2}(u))\\
&=&tW_1(\mu_x^{p_1},\mu_y^{p_1})+(1-t)W_1(\mu_x^{p_2},\mu_y^{p_2}).\label{otherhand}
\end{eqnarray}
Combining (\ref{onehand}) and (\ref{otherhand}), we conclude the function $p\mapsto W(\mu_x^p,\mu_y^p)$ is linear on $[p_1,p_2]$.
$\hfill\Box$\\

We are now in a position to prove the main theorem.\\

{\it Proof of Theorem \ref{thm:1-p}.}
Fix \(p>\frac12\), and let \(\pi\) be any transport plan from \(\mu_x^p\) to \(\mu_y^p\). Recall that
\[
    \mu_x^p(x)=p,
    \qquad
    \mu_y^p(y)=p.
\]
The total source mass outside \(x\) is \(1-p\). Therefore at most \(1-p\) units of mass can arrive at \(y\) from vertices different from \(x\), and hence
\begin{align*}
    \pi(x,y)\geq \mu_y^p(y)-\sum_{u\neq x}\mu_x^p(u)=p-(1-p)=2p-1>0.
\end{align*}

Choose an optimal transport plan \(\pi\) and an optimal Kantorovich potential \(\phi\). Lemma \ref{lem:slackness} gives
\begin{equation}\label{eq:phi_rho}
\phi(x)-\phi(y)=\rho(x,y).
\end{equation}
Consequently, the same potential is optimal at \(p=1\), because $\delta_x=\mu_x^1$, $\delta_y=\mu_y^1$ and
\[
    W_1(\delta_x,\delta_y)
    =
    \rho(x,y)
    =
    \phi(x)-\phi(y)=\sum_{u\in V}\phi(u)(\mu_x^1(u)-\mu_y^1(u)).
\]
By Lemma~\ref{lem:linear}, the function $p\mapsto W_1(\mu_x^p,\mu_y^p)$ is linear on \([p,1]\).

This holds for every \(p>\frac12\).
By continuity, the function $p\mapsto W_1(\mu_x^p,\mu_y^p)$ is linear on $[1/2,1]$.
Hence \(p\mapsto \kappa_p(x,y)\) is linear on \([1/2,1]\), since $\rho(x,y)$ is independent of $p$.

Finally, \(\kappa_1(x,y)=0\). Therefore the linear function on this interval has the form
\[
    \kappa_p(x,y)=C(1-p).
\]
Its coefficient is precisely
\[
    C=
    \lim_{p\to 1-0}
    \frac{\kappa_p(x,y)}{1-p}
    =
    \kappa_{\mathrm{LLY}}(x,y),
\]
which proves $\kappa_p(x,y)=(1-p)\kappa_{\rm LLY}(x,y)$, $p\in[1/2,1]$. Also
$$C=\lim_{p\rightarrow \frac{1}{2}+0}\frac{\kappa_p(x,y)}{1-p}=2\kappa_{\frac{1}{2}}(x,y).$$
This completes the proof of the theorem. $\hfill\Box$

\begin{remark}\label{rmk:sharp}
    The threshold $1/2$ is optimal. Consider the graph consisting of two vertices \(x,y\) joined by an edge of weight \(a>0\). Set
\[
    \mu_x^p=p\delta_x+(1-p)\delta_y,\,\,\,
    \mu_y^p=p\delta_y+(1-p)\delta_x.
\]
Set the distance function $\rho(x,y)=a$. Hence
\(W_1(\mu_x^p,\mu_y^p)
    =a|2p-1|,\)
and therefore
\[
    \kappa_p(x,y)
    =1-|2p-1|=\left\{\begin{array}{lll}
    2-2p&{\rm if}& p\in[1/2,1]\\[1.2ex]
    2p &{\rm if}& p\in[0,1/2).
    \end{array}\right.
\]
For any $c<1/2$, the function $p\mapsto \kappa_p(x,y)$ is not linear on the interval \([c,1]\).
\end{remark}

\begin{remark}
   Theorem~\ref{thm:1-p} together with its proof yields a direct proof of the following useful characterization of the Lin--Lu--Yau curvature in terms of the graph Laplacian, due to M\"unch and Wojciechowski \cite[Theorem~2.1]{MW}: for any two vertices \(x,y\in V\),
 \begin{equation}\label{eq:MW}
     \kappa_{\mathrm{LLY}}(x,y)
=
\inf_{\substack{f\in {\rm Lip}\,1\\
f(x)-f(y)=\rho(x,y)}}
\frac{\Delta f(y)-\Delta f(x)}{\rho(x,y)},
 \end{equation}
where the Laplacian is defined by $\Delta f(x):=\sum_{z\in V}(f(z)-f(x))\mu_x^*(z)$. Indeed, by \eqref{eq:W1} and \eqref{eq:phi_rho}, for any $p\in [1/2,1]$ and any two vertices $x,y\in V$,
\begin{equation*}
    \kappa_p(x, y)=1-\frac{W_1(\mu_x^p,\mu_y^p)}{\rho(x,y)}=\inf_{\substack{f\in {\rm Lip}\,1\\
f(x)-f(y)=\rho(x,y)}}\frac{\rho(x,y)-\sum_{z\in V}f(z)\mu^p_x(z)+\sum_{z\in V}f(z)\mu^p_y(z)}{\rho(x,y)}.\end{equation*}
Using $\rho(x,y)=f(x)-f(y)$ and \(\sum_{z\in V}f(z)\mu_x^p(z)=(1-p)\Delta f(x)+f(x)\), we obtain
\[    \kappa_{p}(x,y)
=(1-p)
\inf_{\substack{f\in {\rm Lip}\,1\\
f(x)-f(y)=\rho(x,y)}}
\frac{\Delta f(y)-\Delta f(x)}{\rho(x,y)}.\]
Combining this identity with Theorem \ref{thm:1-p} immediately yields \eqref{eq:MW}.
\end{remark}

\section{Ricci curvature flow}

In this section, we discuss the $p$-Ollivier curvature flow \cite{Ollivier} and the Lin--Lu--Yau curvature flow \cite{Bai-Lin-Lu-Wang-Yau}
on finite weighted graphs, which can be applied to community detection \cite{Ni-Lin-Luo-Gao,Lai-Bai-Lin,Ma-Yang1}, core detection
\cite{Zhao-Ma-Yang-Zhao}, and other problems on complexity networks.\\

Now we give two general $p$-Ollivier curvature flow on finite weighted graphs. One is the following theorem, which can be derived similarly as in the proof of \cite[Theorem 2.1]{Ma-Yang2}.
\begin{theorem}\label{thm:4}
Let $G=(V,E,\vec{w})$ be a connected finite weighted graph, where $V=\{x_1,\cdots,x_n\}$,
$E=\{e_1,\cdots,e_m\}$, and $\vec{w}=(w_{e_1},\cdots,w_{e_m})\in\mathbb{R}^m_+=\{(z_1,\cdots,z_m)\in\mathbb{R}^m: z_i>0, i=1,\cdots,m\}$. Assume that $p\in [0,1]$, $\mu_x$ and $\mu_{{x}}^\ast$ are two probability measures on $V$
and $V\setminus\{x\}$ respectively as in (\ref{mpx}), and
$\kappa_p(x,y)$ is the $p$-Ollivier curvature given by (\ref{Ollivier}), where $\rho$ is a distance function on $V\times V$.
Let $\mu_{x}^\ast$ and $\rho(x,y)$ be uniquely determined by $\vec{w}$. Suppose the following hypotheses hold:\\[1.2ex]
$(i)$ $A^{-1}\min_{e\in E}w_e\leq \rho(x,y)\leq Aw_{xy}$ for some constant $A$ and all $xy\in E$;\\[1.2ex]
$(ii)$ $\rho(x,y)$ is locally Lipshitz in $\vec{w}\in \mathbb{R}^m_+$ for all $x,y\in V$;\\[1.2ex]
$(iii)$ $\mu_x^\ast(y)$ is locally Lipshitz in $\vec{w}\in \mathbb{R}^m_+$ for all $x\not=y$.\\[1.2ex]
 Then the $p$-Ollivier curvature flow
 \begin{equation}\label{flow}\left\{\begin{array}{lll}
 \frac{d}{dt}w_{xy}(t)=-\kappa_p(x,y)\rho(x,y)\\[1.2ex]
 w_{xy}(t)>0\\[1.2ex]
 w_{xy}(0)=w_{0,xy}>0,\,\,\forall xy\in E
 \end{array}\right.\end{equation}
 has a unique global solution $(w_{xy}(t))_{xy\in E}$ on $t\in[0,+\infty)$.
\end{theorem}

{\it Proof}.
{\it Local existence and uniqueness}\\

In view of the hypotheses $(i)$-$(iii)$, for any compact subset $\overline{\Omega}\subset\mathbb{R}^m_+$ and any
$\vec{w}$, $\vec{\tilde{w}}\in \overline{\Omega}$,
there exists a constant $\Lambda>0$ depending only on $\overline{\Omega}$, $G$ and $A$ such that the following hold:\\[1.2ex]
$(a)$ $\Lambda^{-1}\min_{e\in E}w_e\leq \rho(x,y)\leq \Lambda \sum_{\tau\in E}w_{\tau}$ for all $x,y\in V$, $y\not=x$;\\[1.2ex]
$(b)$ $|\rho(x,y)-\tilde{\rho}(x,y)|\leq \Lambda|\vec{w}-\vec{\tilde{w}}|$ for all $x,y\in V$, where $\tilde{\rho}$ is determined by
$\vec{\tilde{w}}$;\\[1.2ex]
$(c)$ $|\mu_x^\ast(z)-\tilde{\mu}_x^\ast(z)|\leq\Lambda|\vec{w}-\vec{\tilde{w}}|$ for all $x,z\in V$, $z\not=x$,
where $\tilde{\mu}_x^\ast$ is determined by
$\vec{\tilde{w}}$.\\[1.2ex]
Using the same argument as in the proof of \cite[Lemma 3.1]{Ma-Yang2}, we have
\begin{equation}\label{W}
|W_1(\mu_x^p,\mu_y^p)-W_1(\tilde{\mu}_x^p,\tilde{\mu}_y^p)|\leq C|\vec{w}-\vec{\tilde{w}}|, \quad\forall x,y\in V,
\end{equation}
where $C$ is a constant depending only on $\overline{\Omega}$, $G$, $\Lambda$ and $p$. Note that
\begin{equation}\label{k-repres}\kappa_p(x,y)\rho(x,y)=\left(1-\frac{W_1(\mu_x^p,\mu_y^p)}{\rho(x,y)}\right)\rho(x,y)
=\rho(x,y)-W_1(\mu_x^p,\mu_y^p).\end{equation}
This together with $(b)$ and (\ref{W}) implies $\kappa_p(x,y)\rho(x,y)$ is Lipshitz in $\vec{w}\in \overline{\Omega}$. Then the theory of
ODE yields the local existence and uniqueness of solutiona to (\ref{flow}).\\

{\it Global existence}\\

By the definition of the Wasserstein distance and $(a)$, we have
$$W_1(\mu_x^p,\mu_y^p)\leq \max_{u,v\in V}\rho(u,v)\leq \Lambda\sum_{\tau\in E}w_\tau,$$
which together with (\ref{k-repres}) and the hypothesis $(i)$ yields
\begin{equation}\label{e-1}-Aw_{xy}\leq-\rho(x,y)\leq -\kappa_p(x,y)\rho(x,y)\leq W_1(\mu_x^p,\mu_y^p)\leq \Lambda\sum_{\tau\in E}w_\tau,\end{equation}
where $xy$ is an edge. Hence, along the flow (\ref{flow}), we have
$$\label{bdd}-Aw_{xy}(t)\leq \frac{d}{dt}w_{xy}(t)\leq \Lambda\sum_{\tau\in E}w_{\tau}(t).$$
As a consequence, along the flow (\ref{flow}),
$$\label{e-2}w_{0,xy}\exp(-At)\leq w_{xy}(t)\leq\left(\sum_{\tau\in E}w_{0,xy}\right)\exp(m\Lambda t).$$
This implies the global existence of solutions to the flow. $\hfill\Box$\\

Let $\gamma>1$ be a fixed real number. If $w_{xy}/\rho(x,y)>\gamma$, then the edge $xy$ is removed from the graph. This process is called a {\it $\gamma$-surgery}. We now state another $p$-Ollivier curvature flow as follows:

\begin{theorem}\label{thm:5}
If all assumptions in Theorem \ref{thm:4} hold, then up to $\gamma$-surgeries,
the $p$-Ollivier curvature flow
 \begin{equation}\label{eqn-2}\left\{\begin{array}{lll}
 \frac{d}{dt}w_{xy}(t)=-\kappa_p(x,y)w_{xy}(t)\\[1.2ex]
 w_{xy}(t)>0\\[1.2ex]
 w_{xy}(0)=w_{0,xy}>0,\,\,\forall xy\in E
 \end{array}\right.\end{equation}
 has a unique global solution $(w_{xy}(t))_{xy\in E}$ on $t\in[0,+\infty)$.
\end{theorem}

{\it Proof}. {\it Local existence and uniqueness}\\

Let $xy$ be a fixed edge. Note that
$$\kappa_p(x,y)w_{xy}=\kappa_p(x,y)\rho(x,y)\frac{w_{xy}}{\rho(x,y)}=\left(\rho(x,y)-W_1(\mu_x^p,\mu_y^p)\right)\frac{w_{xy}}{\rho(x,y)}.$$
According to the proof of Theorem \ref{thm:4}, both $\rho(x,y)-W_1(\mu_x^p,\mu_y^p$ and $w_{xy}/\rho(x,y)$ are locally Lipschitz in
$\vec{w}\in\mathbb{R}^m_+$. This ensures the local existence and uniqueness of the solution to (\ref{eqn-2}). \\

{\it Global existence up to surgeries}\\

Similar to the proof of Theorem \ref{thm:4}, up to $\gamma$-surgeries, we have an analog of (\ref{e-1}) as
$$-w_{xy}\leq -\kappa_p(x,y)w_{xy}\leq \Lambda \gamma\sum_{\tau\in E}w_\tau.$$
Thus along the flow and up to $\gamma$-surgeries, we obtain the differential inequality
$$-w_{xy}(t)\leq \frac{d}{dt}w_{xy}(t)\leq \Lambda\gamma\sum_{\tau\in E}w_{\tau}(t)$$
and the estimation
$$w_{0,xy}\exp(-t)\leq w_{xy}(t)\leq\left(\sum_{\tau\in E}w_{0,xy}\right)\exp(m\Lambda\gamma t).$$
This yields the global existence of the solution to (\ref{eqn-2}), up to $\gamma$-surgeries. $\hfill\Box$

\begin{remark}
   Let $G=(V,E,\vec{w})$ be a connected finite weighted graph, $V=\{x_1,\cdots,x_n\}$, $E=\{e_1,\cdots,e_m\}$,
and $\vec{w}=(w_{e_1},\cdots,w_{e_m})\in\mathbb{R}^m_+$.
In \cite{Bai-Lin-Lu-Wang-Yau}, Bai et al considered the Lin--Lu--Yau curvature flow
$$\frac{d}{dt}w_{xy}(t)=-\kappa_{\rm LLY}(x,y)w_{xy}(t), \,\,w_{xy}(0)=w_{0,xy}.$$
It is very 
involved to prove that $\kappa_{\rm LLY}(x,y)w_{xy}$ is locally Lipschitz in $\vec{w}\in\mathbb{R}^m_+$.
 According to our Theorem \ref{thm:1-p},
$\kappa_{\rm LLY}=2\kappa_{1/2}$. Then the Lin--Lu--Yau curvature flow becomes the $1/2$-Ollivier curvature flow.
Clearly, $1/2$-Ollivier curvature flow is more concise
(see for example Theorem \ref{thm:5} in the above, or \cite{Ma-Yang2}).
\end{remark}

\section*{Acknowledgement}
\noindent
S.L.'s research is supported by the Scientific Research Innovation Capability Support Project for Young Faculty SRICSPYF-ZY2025160 and the National Natural Science Foundation of China No. 12431004.


\end{document}